\documentclass[a4paper]{amsart}
\usepackage{a4wide}
\usepackage[english]{babel}
\usepackage[utf8]{inputenc}
\usepackage{amsmath,amssymb,amsthm}
\usepackage[table]{xcolor}
\usepackage{graphicx,enumerate,url}
\usepackage{footmisc}

\usepackage{tikz}
\usetikzlibrary{calc,through,backgrounds,shapes,matrix,decorations.markings,decorations.pathmorphing,arrows.meta}
\usepackage{caption}

\usepackage{array}
\usepackage{multirow}
\usepackage{multicol}
\usepackage{hhline}

\usepackage{enumitem}
\setenumerate{label*=\arabic*., itemsep=5pt, topsep=5pt}

\numberwithin{figure}{section}
\numberwithin{equation}{section}

\usepackage[
backend=biber,
style=numeric]{biblatex}
\usepackage{hyperref}
\usepackage{xurl}
\hypersetup{breaklinks=true}

\usepackage{url, hypcap}
\hypersetup{colorlinks=true, citecolor=darkblue, linkcolor=darkblue}

\usepackage{csquotes}

\usepackage{cleveref}

\newtheorem{theorem}{Theorem}[section]
\newtheorem{proposition}[theorem]{Proposition}
\newtheorem{lemma}[theorem]{Lemma}
\newtheorem{corollary}[theorem]{Corollary}

\theoremstyle{definition}
\newtheorem{example}[theorem]{Example}
\newtheorem{remark}[theorem]{Remark}

\newtheorem*{open problem}{Open Problem}

\newcommand{\ZZ}{\mathbb{Z}}
\newcommand{\NN}{\mathbb{N}}

\newcommand{\set}[2]{\left\{ #1 \;|\; #2 \right\}}
\newcommand{\bigset}[2]{\left\{ #1 \;\big|\; #2 \right\}}

\newcommand{\Sym}{{\mathfrak{S}}}
\newcommand{\Perm}{\operatorname{Sym}}
\newcommand{\des}{\operatorname{des}}
\newcommand{\Des}{\operatorname{Des}}

\newcommand{\id}{\operatorname{id}}

\newcommand{\Chow}{\operatorname{\underline{H}}}
\newcommand{\Chowaug}{\operatorname{H}}

\newcommand{\Eulerian}{\operatorname{E}}

\newcommand{\rk}{\operatorname{rk}}

\newcommand{\Dfn}[1]{\emph{\bfseries #1}}

\definecolor{darkblue}{rgb}{0,0,0.7}

\definecolor{lightblue}{rgb}{0.68,0.85,1}

\definecolor{lightgrey}{rgb}{0.9,0.9,0.9}

\definecolor{grey}{rgb}{0.5,0.5,0.5}

\definecolor{darkgreen}{RGB}{0,128,0}

\def\cB{\mathcal{B}}

\def\cD{\mathcal{D}}
\def\cE{\mathcal{E}}
\def\cF{\mathcal{F}}

\def\cL{\mathcal{L}}

\def\cS{\mathcal{S}}

\title[The Chow and augmented Chow polynomials of Uniform Matroids]{The Chow and augmented Chow polynomials \\ of Uniform Matroids}

\author[E.~Hoster]{Elena Hoster}
\address[E.~Hoster]{Fakultät für Mathematik, Ruhr-Universität Bochum, Germany}
\email{elena.hoster@rub.de}

\begin{document}

\begin{abstract}
    We provide explicit combinatorial formulas for the Chow polynomial and for the augmented Chow polynomial of uniform matroids, thereby proving a conjecture by Ferroni.
    These formulas refine existing formulas by \Citeauthor{Hampe_2017} and by \Citeauthor{eur2023stellahedralgeometrymatroids}, offering a combinatorial interpretation of the coefficients based on Schubert matroids.
    As a byproduct, we count Schubert matroids by rank, number of loops, and cogirth.
\end{abstract}

\maketitle
\section{Introduction and main results}
\label{sec: intro}

The Chow polynomial $\Chow_M(x)$ and the augmented Chow polynomial $\Chowaug_M(x)$ of a matroid $M$ are the Hilbert-Poincaré series of the Chow ring and of the augmented Chow ring, respectively. 
These polynomials are known to be unimodal, palindromic, and $\gamma$-positive,
as proven in~\cite{Ferroni2024}.
Both polynomials are also conjectured to be real-rooted, which is only known for the augmented Chow polynomials of \emph{uniform matroids} \cite{Ferroni2024}.
A second proof of the $\gamma$-positivity, along with a combinatorial formula for the~$\gamma$-expansion of both~$\Chow_M(x)$ and~$\Chowaug_M(x)$, 
is provided in \cite{stump2024chowaugmentedchowpolynomials}.

\medskip
Let $U_{k,n}$ be the uniform matroid on ground set $[n] = \{1,\dots , n\}$ and of rank $k$, 
with bases being all $k$-element subsets.
As shown in \cite[Theorem 1.11]{Ferroni2024}, the coefficients of the Chow polynomial and of the augmented Chow polynomial of an arbitrary loopless matroid~$M$ are term-wise maximized when $M$ is a uniform matroid.

\medskip
This paper proves a conjecture by Ferroni regarding the coefficients in the case of uniform matroids and provides a monomial expansion for their Chow and augmented Chow polynomials. 
\begin{theorem}[{\Cite[Conjecture]{Priv_Comm_Ferroni}}]
    \label{thm: Ferroni}
    The Chow polynomial and the augmented Chow polynomial of the uniform matroid $U_{k,n}$ are given by
    \begin{align*}
    \Chow_{U_{k,n}} (x)
    &=
    \sum_{m = 0}^{k-1}
    \# \big\{ \substack{\text{loopless Schubert matroids on the ground set } [n] \\ \text{of rank $m+1$ and cogirth greater than } n-k } \ \big\}
    \cdot x^{m} \, \\
     \Chowaug_{U_{k,n}} (x)
    &=
    \sum_{m = 0}^{k}
    \# \big\{ \substack{\text{\phantom{loop} Schubert matroids\phantom{less} on the ground set } [n] \\ \text{of rank \phantom{+}$m$\phantom{1} and cogirth greater than } n-k } \ \big\} 
    \cdot x^{m} \,.
    \end{align*}
\end{theorem}

\Citeauthor{Hampe_2017} showed in \Cite{Hampe_2017} that the Chow polynomial of the Boolean matroid~$U_{n,n}$ is the~$h$-vector of the permutahedron, the \emph{$n$-th Eulerian polynomial} $A_n(t)$.
Moreover, he showed that the~\emph{$k$-th Eulerian number} counts loopless Schubert matroids on $n$ elements of rank $k$.
Likewise, \Citeauthor{eur2023stellahedralgeometrymatroids} proved in \Cite{eur2023stellahedralgeometrymatroids} that the augmented Chow polynomial of~$U_{n,n}$ is the~$h$-vector of the stellahedron, the~\emph{$n$-th binomial Eulerian polynomial~$\widetilde{A}_n(x)$}.
Here, the coefficients count Schubert matroids of corresponding rank but are not necessarily loopless. 
\Cref{thm: Ferroni} refines both of these representations.

In \Cite{Hameister_2021}, \Citeauthor{Hameister_2021} give a combinatorial formula for the Chow polynomial of the~{$q$-uniform} matroid~$U_{k,n}(q)$ on ground set $[n]$, the $q$-analog of the uniform matroid which becomes~$U_{k,n}$ for~$q = 1$.
This formula is particularly useful in the cases~$k=n$ and~$k = n-1$.
It provides~${x\cdot \Chow_{U_{n-1,n}}(x) = d_n(x)}$, where~$d_n(x)$ denotes the~$n$-th \emph{derangement polynomial}.
More recently, \Citeauthor{liao2024chowringsaugmentedchow} extended this result to the augmented Chow polynomial in \Cite[Theorem 4.7]{liao2024chowringsaugmentedchow}, giving~$\Chowaug_{U_{n-1,n}}(x) = A_n(x)$.

\medskip
We prove \Cref{thm: Ferroni} using combinatorial arguments.
In particular, we provide a combinatorial formula on the number of Schubert matroids on a fixed ground set, according to their \emph{rank}, \emph{cogirth}, and number of \emph{loops}.

\subsection{Main results}
For any nonempty subset $I \subseteq \{1,\dots,  n\}$, consider the disjoint partition of~$I = I_1 \cup \dots \cup I_s$ into maximal consecutive subsets such that $\min(I_j) < \min(I_{j+1})$.
Define the multinomial coefficient
\[
\binom{n}{\Delta I} = \binom{n}{\min(I_1) - 1, \min(I_2)-\min(I_1), \min(I_3)-\min(I_2), \dots , \min(I_s)-\min(I_{s-1})} \,,
\]
which takes the differences of these minima of adjacent subsets.
For $I=\{\}$ being the empty set, we set $\binom{n}{\Delta I} = 1$.

\begin{example}
    The set $I = \{2,3,5,7,8\}$ is partitioned into $\{2,3\}\cup\{5\}\cup\{7,8\}$, the multinomial coefficient is
    \[
    \binom{n}{\Delta I} = \binom{n}{2-1, 5-2, 7-5} = \binom{n}{1,3,2} = \frac{n!}{1!\ 3!\ 2!\ (n-6)!} \,
    \]
    for $n \geq 8$.
\end{example}

A \Dfn{descent} in a sequence of integers $a=(a_1,\dots , a_m)$ is a position $i$ such that $a_i > a_{i+1}$. The \Dfn{descent set} of $a$ is the set of all such positions,~$\Des(a)=\set{i\in \{1,\dots, m-1\}}{ a_i > a_{i+1} }$. Its size gives the number of descents, denoted by~$\des(a)$.
This notation also applies to permutations when written in one-line notation.
Let~$\Eulerian(m,D)$ denote the number of permutations in~$\Sym_m$ that have descent set~$D$.\\
Define $\mathsf{nc}(m)$ to be the set of all subsets of $\{1,\dots , m\}$ that contain \textbf{n}o \textbf{c}onsecutive integers.

\begin{theorem}
    \label{thm: Chow four ways}
    The Chow polynomial of the uniform matroid $U_{k,n}$ is given by any of the following expansions:
    \begin{align}
        \Chow_{U_{k,n}} (x) 
        &= \quad
        \sum_{ \substack{I \subseteq \{1,\dots,k\} \\ 1 \in I } } \ \binom{n}{\Delta I} \ x^{\vert I \vert - 1} \label{eq: Chow monomial} \\
        &=
        \quad
        \sum_{ \substack{ D \in \mathsf{nc}(k-1) \\ 1\notin D } } \quad \Eulerian(n,D) \ \ \cdot x^{\vert D \vert } \cdot (1 + x)^{k - 1 -2\cdot \vert D \vert } \label{eq: Chow Eulerian gamma}    \\
        &= 
        \sum_{ \substack{ \sigma \in \Sym_k \\ \Des(\sigma) \in \mathsf{nc}(k-1) \\ \sigma(1) < \sigma(2) } } 
        \binom{n - \sigma(k)}{k - \sigma(k)} \cdot x^{\des(\sigma)} 
        \cdot (1 + x)^{k - 1 -2\cdot \des(\sigma)} \label{eq: Chow binom gamma} \,.
    \end{align}
\end{theorem}

We deduce \Cref{eq: Chow Eulerian gamma,eq: Chow binom gamma} from the $\gamma$-expansion given in \Cite[Theorem 1.1]{stump2024chowaugmentedchowpolynomials}.
We then prove \eqref{eq: Chow monomial} by comparing the coefficients with \eqref{eq: Chow Eulerian gamma}.

\medskip
In the same way, we get a similar version of \Cref{thm: Chow four ways} for the augmented case. 

\begin{theorem}
    \label{thm: Chow aug four ways}
    The augmented Chow polynomial of the uniform matroid $U_{k,n}$ is given by any of the following expansions:
    \begin{align}
        \Chowaug_{U_{k,n}} (x)
        &= \quad \sum_{ \substack{I \subseteq \{1,\dots,k\} } } \ \binom{n}{\Delta I} \ x^{\vert I \vert } \label{eq: Chow aug monomial} \\
        &=
        \quad 
        \sum_{ \substack{ D \in \mathsf{nc}(k-1) } } \quad \Eulerian(n,D) \quad \cdot x^{\vert D \vert } \cdot (1 + x)^{k -2\cdot \vert D \vert } \label{eq: Chow aug Eulerian gamma} \\
        &=
        \sum_{ \substack{ \sigma \in \Sym_k \\ \Des(\sigma) \in \mathsf{nc}(k-1) } }
        \binom{n - \sigma(k)}{k - \sigma(k)} \cdot x^{\des(\sigma)} 
        \cdot (1 + x)^{k -2\cdot \des(\sigma)} \nonumber \,.
    \end{align}
\end{theorem}

The $\gamma$-expansions given in \eqref{eq: Chow Eulerian gamma} and \eqref{eq: Chow aug Eulerian gamma} are also covered in \Cite{liao2024equivariantgammapositivitychowrings}.

\begin{remark}[Multivariate analogues]
    \label{rem: multi version}
    In \Cref{sec: proofs}, we define the multivariate Chow polynomial and the multivariate augmented Chow polynomial as multivariate versions of $\Chow_{M}(x)$ and of~$\Chowaug_{M}(x)$.
    We prove \Cref{thm: Chow four ways,thm: Chow aug four ways}
    by proving their multivariate analogs.
\end{remark}

We show in \Cref{sec: proofs}, how to translate \Cref{thm: Ferroni} into \eqref{eq: Chow monomial} and \eqref{eq: Chow aug monomial}.
Thereby, 
\begin{itemize}
    \item $I \subseteq [n]$ indexes a set of Schubert matroids,
    \item $1\in I$ ensures that these are loopless, and 
    \item $\max (I) \leq k$ ensures that their cogirth is greater than $n-k$.
\end{itemize}

\begin{example}
    We compute the Chow polynomial of $U_{3,5}$ using all given ways in \Cref{thm: Chow four ways}. 
    Using \eqref{eq: Chow monomial}, we get 
    \begin{align*}
        \Chow_{U_{3,5}}(x) 
        =&
        \binom{5}{0} \cdot x^{\vert \{1\} \vert -1} 
        + \binom{5}{0} \cdot x^{\vert \{1,2\} \vert -1}
        + \binom{5}{0,2} \cdot x^{\vert \{1,3\} \vert -1}
        + \binom{5}{0} \cdot x^{\vert \{1,2,3\} \vert -1} \\
        =&
        1 + x + 10x + x^2 \,, \\
        \intertext{using \eqref{eq: Chow Eulerian gamma}, we get }
        \Chow_{U_{3,5}}(x) 
        =&
        \Eulerian(5,\{ \}) \cdot x^{ 0 } \cdot (1+x)^{ 2-2\cdot 0 }
        + \Eulerian(5,\{ 2 \}) \cdot x^{ 1 } \cdot (1+x)^{ 2-2\cdot 1 } \\
        =&
        (1+x)^2 +  9x \,, \\
        \intertext{and \eqref{eq: Chow binom gamma} finally gives}
        \Chow_{U_{3,5}}(x) 
        =&
        \underbrace{\binom{5 - 3}{3 - 3} \cdot x^{0} \cdot (1+x)^{2} }_{\sigma = 123}
        + \underbrace{\binom{5 - 2}{3 - 2} \cdot x^{1} \cdot (1+x)^{0}}_{\sigma = 132}
        + \underbrace{\binom{5 - 1}{3 - 1} \cdot x^{1} \cdot (1+x)^{0}}_{\sigma = 231} \\
        =&
        \binom{2}{0} \cdot (1+x)^2 
        + \binom{3}{1} \cdot x 
        + \binom{4}{2} \cdot x \\
        =&
        (1+x)^2 + 3x + 6x \,.
    \end{align*}
    All three polynomials coincide with~$\Chow_{U_{3,5}}(x) = 1 + 11x + x^2$.
\end{example}

The following corollary was proven in \Cite{Ferroni2024}. We reprove the results by ordering the sets~$I$ within the monomial expansions \eqref{eq: Chow monomial} and \eqref{eq: Chow aug monomial}, according to the minimum of their last maximal consecutive subset $I_s$, where $I=I_1 \cup \dots \cup I_s$ is the partition into maximal consecutive subsets.

\begin{corollary}[{\Cite[Theorem 1.9]{Ferroni2024}}] 
    \label{cor: chow via derangements}
    We have 
    \begin{align*}
        \Chow_{U_{k,n}}(x) 
          &= \phantom{1+ x\cdot } \sum_{j = 0}^{k-1} \binom{n}{j} \ d_j(x) \ (1 + x + \dots + x^{k-1-j}) \,, \\
        \Chowaug_{U_{k,n}} (x) 
          &= 1+ x \cdot \sum_{j = 0}^{k-1} \binom{n}{j} \ A_j(x) \ (1 + x + \dots + x^{k-1-j}) \,,
    \end{align*}
    where $d_j(x)$ is the $j$-th derangement polynomial, and $A_j(x)$ is the $j$-th Eulerian polynomial.
\end{corollary}

Note that \Cref{thm: Chow four ways,thm: Chow aug four ways} hold at the level of multivariate polynomials as already mentioned in \Cref{rem: multi version}.
An analogue of \Cref{cor: chow via derangements} does not seem to hold.

\section{Background}
\label{sec: background}

\subsection{\texorpdfstring{$R$}{R}-labeling}
Let $P = (P, \leq)$ be a finite graded poset of rank $n$ with minimal element $\widehat{0}$ and maximal element $\widehat{1}$.
The set of cover relations (or edges) of $P$ is denoted by $\cE(P)$.
A maximal chain in $P$ is a set $\cF = \{ \widehat{0} = F_0 < F_1 < \dots < F_n = \widehat{1} \}$. 
For a labeling $\lambda : \cE(P) \to \ZZ$ on the set of edges of $P$ and for a maximal chain $\cF$, 
let $\lambda_\cF = (\lambda_1,\dots , \lambda_n)$ with $\lambda_i = \lambda( F_{i-1} < F_i )$ be its edge-labeling sequence . 
We call $\lambda$ an \Dfn{$R$-labeling} 
if every interval of $P$ admits a unique maximal chain $\cF$ with strictly increasing edge-labeling sequence $\lambda_\cF = ( \lambda_1 < \lambda_2 < \dots < \lambda_n)$.

\subsection{Matroids}
    
    A \Dfn{matroid} $M$ consists of a finite set $E$ (called \Dfn{ground set}) and a collection $\cB$ of subsets of $E$ (called \Dfn{bases}), satisfying the following two properties:
    \begin{itemize}
        \item $\cB$ is nonempty, and
        \item if~$B_1,B_2\in \cB$ are bases and if~$x\in B_1\setminus B_2$, 
        then there is an element~$y \in B_2 \setminus B_1$ such that~$(B_1 \setminus \{x\}) \cup \{y\} \in \cB$ is a basis.
    \end{itemize}
    Every basis has the same size, called the \Dfn{rank} of $M$ denoted by $\rk(M)$.
    A subset $I \subseteq E$ is called \Dfn{independet} if it is a subset of a basis. 
    Otherwise, it is called \Dfn{dependent}.
    The \Dfn{dual matroid}~$M^\ast$ of $M$ is the matroid on the same ground set $E$ and with bases~$\set{ B^c = E\setminus B }{ B \text{ basis of } M}$.
    
    An element $x\in E$ is a \Dfn{loop} in $M$ if it is not contained in any basis.
    A \Dfn{coloop} in $M$ is a loop in the dual matroid $M^\ast$.
    
    A \Dfn{circuit} $C\subseteq E$ is a minimal dependent set of $M$, that is, $C \setminus \{x\}$ is independent for all $x\in C$. 
    The \Dfn{girth} of $M$ is the cardinality of the smallest circuit,
    and the \Dfn{cogirth} of $M$ is the girth of $M^\ast$.
    
    The \Dfn{rank} of a subset $S \subseteq E$ is $\rk_M(S) = \max\set{\# I }{ I \subseteq S \text{ independent  in } M}$, the size of the largest independent subset contained in~$S$.
    If $\rk_M(S \cup \{x\}) > \rk_M(S)$ for every $x \in E \setminus S$, we call~$S$ a \Dfn{flat}.
    The set of flats of $M$ ordered by inclusion forms a lattice, the \Dfn{lattice of flats} denoted by~$\cL(M)$.
    
    \medskip
    Let $[n] = \{1,\dots , n\}$, and let $\binom{[n]}{k}$ be the set of all subsets of $[n]$ that have size~$k$. 

\subsubsection{Uniform matroids}
    For $n\geq k$, the \Dfn{uniform matroid} $U_{k,n}$ is the matroid on the ground set $[n]$ with set of bases $\binom{[n]}{k}$.
    Its lattice of flats $\cL(U_{k,n})$ consists of all subsets of size smaller than~$k$ together with the maximal element $[n]$, ordered by inclusion.
    
\subsubsection{Schubert matroids}
    We use the definition of Schubert matroids given in \Cite{Ferroni_2023} for the fixed ground set $[n]$.
    For~$\pi \in \Sym_n$, let $\leq_{\pi}$ be the total order on $[n]$ given by the one-line notation of~$\pi$, that is
    \[ \pi(1) \leq_\pi \pi(2) \leq_\pi \dots \leq_\pi \pi(n) \,. \]
    If~$\pi = \id$ is the identity permutation, this gives the usual total order $1 < 2 < \dots < n$.
    For subsets~$I = \{i_1 < \dots < i_k\}, J=\{j_1 < \dots < j_k\} \subseteq [n]$ with~$\vert I \vert = \vert J \vert$, 
    we write
    \[
    I \leq_\pi J \quad \text{ if } i_m \leq_\pi j_m \text{ for each } m \in \{1,\dots , k\} \,.
    \]
    For a set $I \in \binom{[n]}{k}$, and for a permutation $\pi\in \Sym_n$, 
    the \Dfn{Schubert matroid} $\cS_{I,\pi}$ is the matroid on the ground set $[n]$ with bases
    \[
        \bigset{ J \in \binom{[n]}{k} }{ I \leq_\pi J } \,.
    \]
    Schubert matroids are a special class of \emph{lattice path matroids}, sometimes also called \emph{nested matroids}~\cite{Hampe_2017}, \emph{generalized Catalan matroids}~\cite{Bonin_2003} or \emph{shifted matroids}~\cite{ardila2002catalanmatroid}.
    
\section{Proofs of the main results}
\label{sec: proofs}

In this section, we introduce multivariate versions of the Chow polynomial and of the augmented Chow polynomial which become the usual polynomials when all variables are set equal.
We then derive \Cref{thm: Chow four ways} and \Cref{thm: Chow aug four ways} by proving their multivariate analogs.
In \Cref{sec: Schubert matroids}, we study Schubert matroids to translate \Cref{thm: Ferroni} into \eqref{eq: Chow monomial} and \eqref{eq: Chow aug monomial}.
The main tool for this translation is \Cref{thm: SM counting when restricted} in which we count Schubert matroids on a fixed ground set, according to rank, cogirth, and the number of loops.
Finally, we prove \Cref{cor: chow via derangements} and discuss some sequences of coefficients in $\Chow_{U_{k,n}}(x)$ and in $\Chowaug_{U_{k,n}}(x)$.

\subsection{Multivariate Chow and augmented Chow polynomials}

The lattice of flats $\cL(M)$ of a matroid $M$ always admits an $R$-labeling \Cite[Proposition 2.2]{MR0354473}.
Fix such an $R$-labeling $\lambda$, then, by \Cite[Theorem 1.1]{stump2024chowaugmentedchowpolynomials}, the Chow polynomial of $M$ is given by
\begin{equation}
    \label{eq: Chow gamma exp Stump}
    \Chow_{M} (x)
    = \sum_{ \cF } x^{ \des(\lambda_\cF ) } \cdot (1 + x)^{\rk(M) { - 1} -2\cdot \des(\lambda_\cF)} \,,
\end{equation}
where the sum ranges over all maximal chains $\cF$ whose edge-labeling sequence $\lambda_\cF$
has no consecutive descents, $\Des(\lambda_\cF) \in \mathsf{nc}(k-1)$, and with $1\notin \Des(\lambda_\cF)$.

The augmented Chow polynomial of $M$ is given by
\begin{equation}
    \label{eq: augChow gamma exp Stump}
    \Chowaug_{M} (x)
    = \sum_{ \cF } x^{ \des(\lambda_\cF ) } \cdot (1 + x)^{\rk(M) -2\cdot \des(\lambda_\cF)} \,,
\end{equation}
where the sum ranges over all maximal chains $\cF$ whose edge-labeling sequence $\lambda_\cF$
has no consecutive descents, $\Des(\lambda_\cF) \in \mathsf{nc}(k-1)$.

Define the \Dfn{multivariate Chow polynomial} of $M$ by
\begin{align*}
    \Chow_M (\boldsymbol{x})
    =&
    \sum_{ \substack{ \cF \\ 1 \notin \Des(\lambda_F)  } } \left(\prod_{i\in \Des(\lambda_\cF)} x_i \right) \cdot  
    \left(\prod_{ \substack{i \in \{1,\dots , k-1 \} \\ i,i+1 \notin \Des(\lambda_\cF) } } (1+x_i) \right) &&\in \NN [x_1,\dots , x_{\rk(M)-1}]
\intertext{and define the \Dfn{multivariate augmented Chow polynomial} of~$M$ by}
    \Chowaug_M (\boldsymbol{x})
    =&\quad \
    \sum_{ \cF } \quad \left(\prod_{i\in \Des(\lambda_\cF)} x_i \right) \cdot  
    \left(\prod_{ \substack{i \in \{0,\dots , k-1 \} \\ i,i+1 \notin \Des(\lambda_\cF) } } (1+x_i) \right) &&\in \NN [x_0,\dots , x_{\rk(M)-1}]
\end{align*}
Here, both sums range over all maximal chains $\cF$ in~$\cL(M)$ such that $\Des(\lambda_\cF) \in \mathsf{nc}(k-1)$ contains no consecutive elements.
Setting all $x_i = x$, we get back the usual Chow and augmented Chow polynomials of $M$, respectively, as given in \eqref{eq: Chow gamma exp Stump} and \eqref{eq: augChow gamma exp Stump}.

\begin{remark}[A natural choice for the multivariate version]
    Chow polynomials and augmented Chow polynomials are evaluations of the \emph{Poincaré-extended $\boldsymbol{ab}$-index} \Cite[Theorem 2.6]{stump2024chowaugmentedchowpolynomials},
    which is a polynomial in the variable~$y$ with coefficients in the non-commutative ring $\ZZ\langle \boldsymbol{a},\boldsymbol{b}\rangle$.
    This polynomial was introduced in \Cite{poincareextended}
    and encodes the positions of ascents and descents of edge-labeling sequences of maximal chains.
    Applying the evaluation approach from \Cite{stump2024chowaugmentedchowpolynomials} to the identity given in~\Cite[Corollary 2.11]{poincareextended}, and distinguishing descents by their position, we derive the multivariate forms presented here.
\end{remark}

\subsection{Eulerian numbers and subset permutations}
An $R$-labeling on $\cL(M)$ can be constructed by any total order on the atoms of $\cL(M)$,
see \Cite[Proposition 2.2]{MR0354473}.
For the uniform matroid $U_{k,n}$, the usual total order $1< \dots< n $ yields the $R$-labeling $\lambda$, defined by 
\begin{equation}
    \label{eq: R labeling uniform matroid}
    \lambda(S \prec T) = \min( T\setminus S ) \,.
\end{equation}
In particular, the entries in the sequence $\lambda_\cF = (\lambda_1,\dots , \lambda_k)$ for a maximal chain $\cF$ in $\cL(U_{k,n})$ are all different.

\medskip
A \Dfn{subset permutation} $\sigma\in \Perm(S)$ is a permutation of a subset $S\subseteq [n]$.
For a subset~${S=\{s_1<\dots < s_k\}}$, the one-line notation of $\sigma = \sigma_1 \sigma_2 \dots \sigma_k$ is given by $\sigma_i = \sigma(s_i)$.

\begin{lemma}
    \label{lma: subset permutations in lattice}
    The map $\cF \mapsto \lambda_\cF$ sending maximal chains in $\cL(U_{k,n})$ to their edge-labeling sequence with labels defined in \eqref{eq: R labeling uniform matroid}
    is injective.
    Its image is the following set of subset permutations:
    \begin{align*}
        \bigcup_{\substack{S \subseteq [n] \\ \vert S \vert = k }}
        \bigset{ \sigma \in \Perm(S) }{ \{1,\dots , \sigma(\max S ) \} \subseteq S } \,.
    \end{align*}
\end{lemma}
\begin{proof}
    A maximal chain in $\cL(U_{k,n})$ is of the form
    \[
        \cF = \left\{ \phantom{\big |} \{\} \subset \{\lambda_1\} \subset \{\lambda_1, \lambda_2\} \subset \dots \subset \{\lambda_1,\dots , \lambda_{k-1}\} \subset \{1,\dots , n \} \ \right\}
    \]
    for pairwise distinct integers $\lambda_1,\dots , \lambda_{k-1}$.
    Its edge-labeling sequence is~$\lambda_{\cF} = ( \lambda_1, \lambda_2 , \dots , \lambda_{k-1}, \lambda_k)$ with $\lambda_k = \min \left( [n] \setminus \{\lambda_1 , \dots , \lambda_{k-1} \} \right)$ 
    being determined by the preceding labels.
    This gives the surjection $\cF \mapsto \lambda_\cF$.
    Moreover, the first $k-1$ entries in $\lambda_{\cF}$ of a chain~{$\cF = \{ F_0 \subset \dots \subset F_{k-1} \subset [n]  \}$}
    specify the order in which the elements from $\cF_{k-1}$ are adjoined along $\cF$.\\
    Vice versa, let the first~$k-1$ entries of a permutation~$\sigma\in \Perm(S)$ for a set $S\in \binom{[n]}{k}$ determine the first $k-1$ labels $\sigma_1, \dots , \sigma_{k-1}$ of an edge-labeling sequence $\lambda_{\cF}$. 
    Since the $k$-th entry $\lambda_k$ of this sequence is the smallest integer in $[n]$ which has not appeared yet in $\lambda_{\cF}$, we have $\sigma_k = \lambda_k$ if and only if for every integer $i < \sigma_k$ there is an index $j\in \{1,\dots , k-1\}$ such that~$\sigma_j = i$, and thus~$i \in S$.
\end{proof}

Let $\sigma = \sigma_1 \dots \sigma_k \in \Perm(S)$ be a subset permutation of a subset $S \subseteq [n]$. 
Then, we can
\begin{itemize}
    \item \Dfn{extend} $\sigma$ to the permutation ${\sigma}^{\uparrow} = \sigma_1 \dots \sigma_k t_1 \dots t_{n-k} \in \Sym_n$ with $\{t_1 < \dots < t_{n-k} \} = S^c$, or
    \item \Dfn{standardize} $\sigma$ to the permutation $\sigma_{\downarrow}\in \Sym_k$ by sending $\sigma_i$ to its
    image under the unique order-preserving map $S \to \{1,\dots , k\}$.
\end{itemize}
Observe that both extending and standardizing preserves descents.

\begin{example}
    Let $n = 8$, and let $\sigma = 3641\in \Perm(\{1,3,4,6\})$.
    Then, extending $\sigma$ gives the permutation~${\sigma}^\uparrow = 36412578 \in \Sym_8$, while standardizing $\sigma$ gives~$\sigma_{\downarrow} = 2431 \in \Sym_4$.
\end{example}

\begin{proposition}
    \label{prop: Chow and augChow uniform gamma exp}
    The multivariate Chow polynomial of the uniform matroid $U_{k,n}$ is given by
    \begin{align*}
        \Chow_{U_{k,n}} (\boldsymbol{x})
        &= \quad
        \sum_{ \substack{ D \in \mathsf{nc}(k-1) \\ 1\notin D } } \quad 
        \Eulerian(n,D) \quad \ 
        \cdot \left(\prod_{i\in D} x_i \right) 
        \cdot \left(\prod_{ \substack{i \in \{1,\dots , k-1 \} \\ i,i+1 \notin D } } (1+x_i) \right) \\
        &= 
        \sum_{ \substack{ \sigma \in \Sym_k \\ \Des(\sigma) \in \mathsf{nc}(k-1) \\ \sigma(1) < \sigma(2) } } 
        \binom{n - \sigma(k)}{k - \sigma(k)} 
        \cdot \left(\prod_{i\in \Des(\sigma)} x_i \right) 
        \cdot \left(\prod_{ \substack{i \in \{1,\dots , k-1 \} \\ i,i+1 \notin \Des(\sigma) } } (1+x_i) \right) \,
    \end{align*}
    and the multivariate augmented Chow polynomial of the uniform matroid $U_{k,n}$ is given by
    \begin{align*}
        \Chowaug_{U_{k,n}} (\boldsymbol{x})
        &= \quad
        \sum_{ \substack{ D \in \mathsf{nc}(k-1) } } \quad 
        \Eulerian(n,D) \quad \ 
        \cdot \left(\prod_{i\in D} x_i \right) 
        \cdot \left(\prod_{ \substack{i \in \{0,\dots , k-1 \} \\ i,i+1 \notin D } } (1+x_i) \right) \\
        &= 
        \sum_{ \substack{ \sigma \in \Sym_k \\ \Des(\sigma) \in \mathsf{nc}(k-1) } } 
        \binom{n - \sigma(k)}{k - \sigma(k)} 
        \cdot \left(\prod_{i\in \Des(\sigma)} x_i \right) 
        \cdot \left(\prod_{ \substack{i \in \{0,\dots , k-1 \} \\ i,i+1 \notin \Des(\sigma) } } (1+x_i) \right) \,.
    \end{align*}
\end{proposition}
\begin{proof}
    We apply \Cref{lma: subset permutations in lattice} to the definitions of the multivariate polynomials, and get
\begin{align*}
    \Chow_{U_{k,n}} (x)
    &= \sum_{ \sigma } x^{ \des(\sigma ) } \cdot (1 + x)^{k -1 -2\cdot \des(\sigma)} \,, \quad \text{ and} \\
    \Chowaug_{U_{k,n}} (x)
    &= \sum_{ \sigma } x^{ \des(\sigma ) } \cdot (1 + x)^{k -2\cdot \des(\sigma)} \,,
\end{align*}
where the sums range over all subset permutations~$\sigma = \sigma_1 \dots \sigma_k$, such that $\Des(\sigma) \in \mathsf{nc}(k-1)$, ($1\notin \Des(\sigma)$ in the case of the Chow polynomial), and~$\{1, 2,\dots , \sigma_k\}\subseteq\{\sigma_1, \sigma_2 \dots , \sigma_k\}$.

The subset permutations~$\sigma$ that appear as edge-labeling sequence $\lambda_\cF$ are exactly those whose extended permutation $\sigma^\uparrow$ has descent set
\[
\Des({\sigma}^{\uparrow}) \subseteq \{1,\dots,k-1\} \,.
\]
This is because all integers smaller than $\sigma(k)$ must be positioned within the first $k$ entries of the one-line notation of ${\sigma}^{\uparrow}$, 
meaning there are no descents beyond position $k$ since $\sigma^\uparrow (k) < \dots < \sigma^\uparrow (n)$.
This proves the first equation for each polynomial.

The map $\sigma \mapsto \sigma_{\downarrow}$ is a surjective map from the set of subset permutations $\sigma$, that appear as edge-labeling sequence $\lambda_\cF$, to $\Sym_k$. 
Any integer smaller than $\sigma_k$ must be in $\{\sigma_1,\dots , \sigma_k\}$ and is therefore fixed under standardizing.
The fiber of a permutation $\pi\in \Sym_k$ is thus given by
\[
    \# \left\{\sigma \in \Perm(S) \ {\Big \vert } \  S \in \binom{[n]}{k} \,, \, \sigma_\downarrow = \pi \right\} 
    = \binom{n-\pi(k)}{k - \pi(k)} \,.
\]
Since the descents are preserved under standardizing, this completes the proof.
\end{proof}

\begin{theorem}
    \label{thm: multi Chows monomial}
    The multinomial Chow polynomial and the multinomial augmented Chow polynomial of the uniform matroid $U_{k,n}$ are given by
    \begin{align*}
        \Chow_{U_{k,n}}(\boldsymbol{x}) 
        &= \sum_{\substack{I \subseteq \{1,\dots,k\} \\ 1 \in I }} 
        \binom{n}{ \Delta I } \prod_{i\in I\setminus \{1\}} x_{i-1} \,,
        \quad\text{and} \\
        \Chowaug_{U_{k,n}}(\boldsymbol{x}) 
        &= \sum_{\substack{I \subseteq \{1,\dots,k\} }} 
        \binom{n}{ \Delta I } \ \ \prod_{i\in I} \ x_{i-1} \,.
    \end{align*}
\end{theorem}
\begin{proof}
    For positive integers $k \leq n$, let
    \begin{align*}
        F_{k,n}(\boldsymbol{x}) 
        &= \sum_{\substack{I \subseteq \{1,\dots,k\} \\ 1 \in I }} 
        \binom{n}{ \Delta I } \prod_{i\in I\setminus \{1\}} x_{i-1} \,,
    \quad\phantom{and} \\
        G_{k,n}(\boldsymbol{x}) 
        &= \sum_{\substack{I \subseteq \{1,\dots,k\} }} 
        \binom{n}{ \Delta I } \ \ \prod_{i\in I} \ x_{i-1} \,.
    \end{align*}
    Comparing coefficients, we show that
    \[
        F_{k,n}(\boldsymbol{x}) = \Chow_{U_{k,n}} (\boldsymbol{x})
        \qquad \text{ and } \qquad
        G_{k,n}(\boldsymbol{x}) = \Chowaug_{U_{k,n}} (\boldsymbol{x}) \,.
    \]
    We focus on the proof of the multivariate Chow polynomial, as it has more restrictions on the sum's range. 
    The analogous statements for the multivariate augmented Chow polynomial follows by omitting these restrictions.\\
    Let $S \subseteq \{1,\dots, k-1\}$ be a subset and let $\boldsymbol{x}^S = \prod_{i\in S} x_i $.
    The coefficient of $\boldsymbol{x}^S$ in $F_{k,n}(\boldsymbol{x})$ and in~$\Chow_{U_{k,n}}(\boldsymbol{x})$, respectively, is given by
    \begin{align*}
        [\boldsymbol{x}^S] F_{k,n} (\boldsymbol{x}) = \binom{n}{\Delta \{1\}\cup(S + 1)} 
        \qquad \text{ and } \qquad
        [\boldsymbol{x}^S] \Chow_{U_{k,n}} (\boldsymbol{x}) =  \sum_{ D } \Eulerian (n,D) 
    \end{align*}
    where $S + 1 = \set{s+1}{s \in S}$, 
    and where the sum on the right ranges over all sets $D \in \mathsf{nc}(k-1)$ 
    with $1\notin D$, 
    such that
    \begin{equation}
        \label{eq: write S via D}
        S = D \cup \bigset{ i\in \{1,\dots , k-1\} }{ i,i+1 \notin D }
        = \bigset{i\in \{1,\dots , k-1\}}{ i \in D \text{ or } i \notin \overleftarrow{D}}
    \end{equation}
    where $\overleftarrow{D} = \bigset{i\in \{1,\dots , k-1\}}{i\in D \text{ or } i-1 \in D }$.
    This follows from \Cref{prop: Chow and augChow uniform gamma exp}.
    Let us construct all such subsets $D$.
    If $s,s+1\in S$, we must have $s+1 \notin D$, which also forces $s+1 \notin \overleftarrow{D}$. 
    Thus, for a maximal consecutive subset $\{s, s+1, \dots , s+m\}$ of $S$, none of $s+1, \dots , s+m$ are in~$D$ or~$\overleftarrow{D}$.
    Let~$S = S_1 \cup S_2 \cup \dots \cup S_j$ be the disjoint partition of~$S$ into maximal consecutive subsets, then we just proved that $D \subseteq \{ \min (S_1) ,\dots \min (S_j) \}$.
    Each of these subsets satisfies~\eqref{eq: write S via D}, since~$\min (S_i) \notin D$ for some index $i$ implies $\min (S_i) \notin \overleftarrow{D}$.
    These subsets~$D$ naturally avoid consecutive elements, and the restriction $1\notin D$ does not affect any other constraint.
    We thus get
    \begin{align*}
      [\boldsymbol{x}^S] \Chow_{U_{k,n}} (\boldsymbol{x})
      &=  \sum_{ \substack{D \subseteq \{ \min (S_1) ,\dots \min (S_j) \} \\ 1\notin D } } 
        \Eulerian (n,D) \\
      &= \# \bigset{ w\in \Sym_n }{ \Des(w) \subseteq \{ \min (S_1) ,\dots \min (S_j) \} } \\
      &= \binom{ n }{ \min (S_1) , \min (S_2) - \min (S_1) , \dots , \min (S_{j}) - \min (S_{j-1}) , n - \min (S_j) } \,.
    \end{align*}
    The last equation is explained as follows: To construct a permutation $w\in \Sym_n$ without any descent in the first $m+1$ positions, choose $m$ elements for $w(1)<w(2)< \dots < w(m)$. Each choice uniquely determines the initial sequence. If the next $\ell+1$ positions must also avoid any descent, choose $\ell$ elements of the remaining $n-m$. Repeating this iteratively yields the multinomial coefficient.\\
    Since the corresponding minima in the partition of $ \{1\} \cup (S + 1)$ are $1, \min (S_1) + 1$ (if and only if~$\min (S_1) > 1$), $ \min (S_2) + 1, \dots , \min (S_j) + 1$, and since $(1-1)! = 0! = 1$, we have
    \[
        \binom{ n }{ \min (S_1) , \min (S_2) - \min (S_1) , \dots , \min (S_{j}) - \min (S_{j-1}) , n - \min (S_j) }
        = \binom{ n }{ \Delta \{1\} \cup (S + 1) }
    \]
    which is the coefficient of $\boldsymbol{x}^S$ in $F_{k,n}(\boldsymbol{x})$, as desired.
\end{proof}

\begin{proof}[Proof of \Cref{thm: Chow four ways,thm: Chow aug four ways}]
    This follows by setting $x_i = x$ in \Cref{prop: Chow and augChow uniform gamma exp} and in \Cref{thm: multi Chows monomial}.
\end{proof} 

\subsection{Schubert matroids}
\label{sec: Schubert matroids}
Recall \Citeauthor{Priv_Comm_Ferroni}'s conjecture stated in \Cref{thm: Ferroni}, saying that the coefficient of $x^m$ in the Chow and in the augmented Chow polynomials of the uniform matroid $U_{k,n}$ are given by
\begin{align*}
    [x^m] \Chow_{U_{k,n}}(x)
    &=
    \# \big\{ \substack{\ \text{loopless Schubert matroids on the ground set } [n] \\ \text{of rank $m+1$ and cogirth greater than } n-k }  \ \big\} 
    \quad
    \text{for } 0\leq m \leq k-1 \,,  \\
    [x^m] \Chowaug_{U_{k,n}}(x)
    &=
    \# \big\{ \ \substack{\text{\phantom{loop} Schubert matroids\phantom{less} on the ground set } [n] \\ \text{of rank \phantom{+}$m$\phantom{1} and cogirth greater than } n-k } \ \big\} \quad
    \text{for } 0\leq m \leq k \,.
\end{align*}
In this section, we study the Schubert matroid on the fixed ground set $[n]$ to determine the values of the right-hand side.
\medskip

The Schubert matroid $\cS_{I,\pi}$ for $I \in \binom{[n]}{k}$ and $\pi \in \Sym_n$ has bases $\set{J\in \binom{[n]}{k}}{I\leq_\pi J }$.
The rank of $\cS_{I,\pi}$ is $k$, the size of the set $I$, 
but its loops and cogirth depend on both $I$ and $\pi$.

\begin{proposition}
    \label{prop: Schubert matroid props}
    Let $n\geq k$ be positive integers, and let $I\subseteq \{1,\dots, n\}$ and $\pi\in \Sym_n$.
    Let $\min_\pi (I)$ denote the minimal, and let $\max_\pi (I)$ denote the maximal element in $I$ with respect to $\leq_\pi$.
    \begin{enumerate}
        \item $\cS_{I,\pi}$ has loops $\{\pi(1), \pi(2), \dots , \pi( m - 1) \}$ where $\pi(m) = \min_\pi (I)$.
        \item $\cS_{I,\pi}$ has cogirth $n+1- c$ where $\pi(c) = \max_\pi (I)$.
        \item $\cS_{I,\pi} = \cS_{\pi^{-1}(I),\id}$ with~$\pi^{-1}(I) = \set{\pi^{-1}(i)}{i\in I}$.
    \end{enumerate}
\end{proposition}
\begin{proof}
    An element $x\in [n]$ is a loop in $\cS_{I,\pi}$ if $x$ is smaller than every element in $I$, which proves the first statement.
    The total order $\leq_\pi$ is defined by $i <_\pi j$ if $\pi^{-1}(i) < \pi^{-1}(j)$ in the usual order~$1<2<\dots < n$. The third statement follows immediately by this definition, since $J$ is a basis in $S_{I,\pi}$ if and only if $\pi^{-1}(J)$ is a basis in $S_{\pi^{-1}(I),\id}$. 
    
    To prove the second statement, first consider the identity permutation~$\pi = \id$ .
    The girth of a matroid $M$ gives the size of the smallest possible dependent set in $M$. Any subset of size smaller than the girth is therefore independent, meaning that it is contained in some basis of~$M$.
    Thus, the Schubert matroid~$\cS_{I,\id}$ has cogirth greater than $n-c$ if and only if every subset~$S \subseteq [n]$ of size $\vert S \vert = c$ contains a basis $J$ of $\cS_{I,\id}$.
    In particular, the set $S = \{1, \dots , c\}$ must then contain a subset $J$ with $J \geq I$,
    which gives $I \subseteq \{1,\dots , c\}$.\\
    Vice versa, if $I \subseteq \{1, \dots , c\}$ has size $\vert I \vert = m$ 
    and if~$S = \{s_1 < \dots < s_c \} \subseteq [n]$ is a subset of size~$c$, then let~$J = S \setminus \{s_1,\dots , s_{c-m}\}$ be the set containing the $m$ largest elements of $S$. 
    Since $i \leq s_i$ for all indices $i$, we get 
    \[
        I \leq \{ c -m+1 , \dots , c\} \leq \{ s_{c -m+1} , \dots , s_c\} = J \text{ with } J \subseteq S  \,.
    \]
    This proves, that the cogirth of $\cS_{I,\id}$ is greater than $n-c$ if and only if $\max(I) \leq c$. 
    Together with the contraposition of this statement and with the third statement of our lemma, we have proven the second.  
\end{proof}

\begin{theorem}
    \label{thm: SM counting when restricted}
    The number of Schubert matroids on the ground set $[n]$ of rank $m$, with $\ell$ loops, and having cogirth $n+1-k$ is 
    \[
        \sum_{ \substack{I \subseteq \{\ell + 1,\dots , k\} \\ \ell + 1 , k \in I \\ \vert I \vert = m } } \binom{n}{\Delta I} \,.
    \]
\end{theorem}

Before we prove \Cref{thm: SM counting when restricted},
we deduce \Cref{thm: Ferroni}.

\begin{proof}[Proof of \Cref{thm: Ferroni}]
    This is an immediate consequence of \Cref{thm: SM counting when restricted}, as the polynomials coincide with the monomial expansions given in \Cref{thm: Chow four ways} and \Cref{thm: Chow aug four ways}.
\end{proof}

\begin{proof}[Proof of \Cref{thm: SM counting when restricted}]
By \Cref{prop: Schubert matroid props}, the number of Schubert matroids on the ground set~$[n]$ of rank~$m$, with~$\ell$ loops, and having cogirth~$n-k$ is
\begin{align}
    \label{eq: Chow monomial orbit}
    \sum_{ \substack{I \subseteq \{\ell + 1,\dots,k\} \\ \ell + 1, k \in I \\\vert I \vert = m } } \#\set{\cS_{\pi(I),\pi}}{\pi \in \Sym_n} \,.
\end{align}
The group $\Sym_n$ acts on the set of Schubert matroids by~$\tau \ast \cS_{I,\pi} = \cS_{\tau(I), \tau \pi}$.
We claim, that the multinomial coefficient $\binom{n}{\Delta I}$ gives the size of the orbit~$\set{ \cS_{\pi(I),\pi} }{ \pi \in \Sym_n}$ with respect to this group action.
Let~$I = I_1 \cup \dots \cup I_s \subseteq [n]$ be a nonempty set, written in its disjoint partition into maximal consecutive subsets such that $\min(I_j) < \min(I_{j+1})$.
By the orbit-stabilizer theorem and by Lagrange's theorem, we have
\[
\# \set{ \cS_{\pi(I),\pi} }{ \pi \in \Sym_n} = \frac{n!}{ \# \set{\pi\in \Sym_n}{ \cS_{\pi(I),\pi} = \cS_{I,\id} } } \,.
\]
A permutation $\pi\in \Sym_n$ satisfies $\cS_{\pi(I),\pi} = \cS_{I,\id}$ if and only if it can be written as a product of permutations $\pi = \pi^{(0)} \pi^{(1)} \cdots \pi^{(s)}$ such that
\begin{align*}
    \pi^{(0)} &\in \Perm(\{1,\dots,\min(I_1)-1\}), \\
    \pi^{(j)} &\in \Perm(\{\min(I_j), \min(I_j)+1, \dots, \min(I_{j+1})-1\}) \quad \text{for } 1 \leq j \leq s-1 \,, \\
    \pi^{(s)} &\in \Perm(\{\min(I_{s}),\dots, n \}) \,.
\end{align*} 
A permutation $\pi\in \Sym_n$ of this form certainly satisfies $\cS_{\pi(I),\pi} = \cS_{I,\id}$, since it is satisfied by each adjacent transposition~$(a, a+1)\in \Sym_n$ of this form.
Consider a transposition~$\tau = (a, a+1)$ with~$ a = \min(I_k) -1$ for some~$k\in \{1,\dots , s\}$, which is not of this form.
Since~$\min (I_k)$ is the minimum of a maximal consecutive subset in~$I$, we have~$a\notin I$.
Then,~$\tau (I) < I$, which implies that $\tau (I)$ is not a basis in~$\cS_{I,\id}$, 
and therefore~$\cS_{I,\id} \neq \cS_{\tau (I), \tau}$.
\end{proof}

\begin{example}
    Let $I = \{2,3,5,7,8\}=\{2,3\}\cup\{5\}\cup\{7,8\}$. 
    The stabilizer of the Schubert matroid $\cS_{I,\id}$ on the ground set $[n]$, for $n\geq 8$, is
    \begin{align*}
        \set{\pi\in \Sym_n}{ \cS_{\pi(I),\pi} = \cS_{I,\id} }
        &\cong
        \Perm(\{1\}) \times \Perm(\{2,3,4\}) \times \Perm(\{5,6\}) \times \Perm(\{7,\dots , n\}) \\
        &\cong
        \Sym_1 \times \Sym_3 \times \Sym_2 \times \Sym_{n-6} \,.
    \end{align*}
    Thus, the size of the orbit $\set{ \cS_{\pi(I),\pi} }{ \pi \in \Sym_n}$ equals the multinomial coefficient $\binom{n}{\Delta I}$, that is
    \[
    \# \set{ \cS_{\pi(I),\pi} }{ \pi \in \Sym_n} 
    = \frac{n!}{ \# ( \Sym_1 \times \Sym_3 \times \Sym_2 \times \Sym_{n-6} ) }
    = \frac{n!}{1!\ 3!\ 2!\ (n-6)!}
    =  \binom{n}{\Delta I} \,.
    \]
\end{example}

\subsection{Special cases and sequences of coefficients}
In this section, 
consider the polynomials~$\Chow_{U_{k,n}}(x)$ and~$\Chowaug_{U_{k,n}}(x)$ written in their monomial bases as given in \Cref{thm: Chow four ways} and in \Cref{thm: Chow aug four ways}.

\medskip
The \Dfn{$n$-th derangement polynomial} is defined by $d_n(x)=\sum_{w\in \cD_n} x^{\operatorname{exc}(w)}$ where $\cD_n$ is the set of fixpoint-free permutations in $\Sym_n$ and $\operatorname{exc}(w)=\# \set{i \in [n]}{w(i)>i}$ counts the number of excedances of $w$. 
The \Dfn{$n$-th Eulerian polynomial} is defined by $A_j(x) = \sum_{w \in \Sym_j } x^{\des(w)}$.
For~$j = 0$, the polynomials are set to $d_0(x) = 1 = A_0 (x)$.
For $j\geq 2$, we have 
\[
    x \cdot \Chow_{U_{n-1,n}}(x) = d_n(x)
    \qquad \text{and} \qquad 
    \Chowaug_{U_{n-1,n}}(x) = A_n(x)
\]
as given in \Cite{Hameister_2021} and in \Cite{liao2024chowringsaugmentedchow}, respectively.
For arbitrary $k\leq n$, we aim to provide a new proof for
\begin{align*}
    \Chow_{U_{k,n}}(x) 
    &= \phantom{1+ x\cdot } \sum_{j = 0}^{k-1} \binom{n}{j} \ d_j(x) \ (1 + x + \dots + x^{k-1-j}) \,, \\
    \Chowaug_{U_{k,n}} (x) 
    &= 1+ x \cdot \sum_{j = 0}^{k-1} \binom{n}{j} \ A_j(x) \ (1 + x + \dots + x^{k-1-j}) \,,
\end{align*}
first proven in \Cite[Theorem 1.9]{Ferroni2024}.
\begin{proof}[Proof of \Cref{cor: chow via derangements}]
    For a nonempty set $I = I_1 \cup \dots \cup I_s$ written as its disjoint partition into maximal consecutive subsets, such that $\min (I_j) < \min(I_{j+1})$, 
    let $m(I) = \min(I_s)$.
    We order all subsets $I \subseteq \{1,\dots , k\}$ by this minimum $m(I)$.
    For instance,~$m(I) = 1$ for~$I$ being any consecutive set starting with $1$, and $m(I)=2$ for $I$ being any consecutive set starting with~$2$.
    For the Chow polynomial, we get
    \begin{align*}
        \Chow_{U_{k,n}} (x)
        &=
        \sum_{j = 1}^{k} \sum_{ \substack{I \subseteq \{1,\dots,k\} \\ 1\in I \\ m(I) = j } } \binom{n}{\Delta I} x^{\vert I \vert - 1} \\
        &= (1+x+\dots + x^{k-1}) + 
        \sum_{j = 3}^{k} \ \sum_{ \substack{I \subseteq \{1,\dots,k\} \\ m(I) = j \\ 1\in I } } \binom{n}{ j - 1} \binom{j - 1}{\Delta I\setminus I_s} 
        \ x^{\vert I\setminus I_s \vert - 1} \ x^{\vert I_s \vert } \\
        &= (1+\dots + x^{k-1}) + 
        \sum_{j = 3}^{k} \ \sum_{ \substack{I \subseteq \{1,\dots,j-2\} \\ 1\in I } } \binom{n}{ j - 1} \binom{j - 1}{\Delta I\setminus I_s} 
        \ x^{\vert I \vert - 1} \cdot (x + \cdots + x^{k-j}) \\
        &= (1+\dots + x^{k-1}) + 
        \sum_{j = 2}^{k-1} \binom{n}{ j } \Chow_{U_{j-1,j}}(x) \cdot x\cdot (1 + \cdots + x^{k-1-j}) \,.
    \end{align*}
    Since $d_j(x) = x\cdot \Chow_{U_{j-1,j}}(x)$ for $j \geq 2$, $d_1(x) = 0$ and $d_0(x) = 1$, we get
    \[
        \Chow_{U_{k,n}}(x) 
        l= \sum_{j = 0}^{k-1} \binom{n}{j} \ d_j(x) \ (1 + x + \dots + x^{k-1-j}) \,.
    \]
    The same argumentation for the augmented Chow polynomial gives
    \begin{align*}
        \Chowaug_{U_{k,n}} (x)
        &= 1 + (x + \dots + x^k) + \binom{n}{1} \cdot (x + \dots + x^{k-1}) 
        + \sum_{j = 2}^{k-1} \binom{n}{j} \Chowaug_{U_{j-1,j}}(x) \cdot (x + \dots + x^{k-j}) \,.
    \end{align*}
    Here, the constant term $1$ is produced by $I = \{\}$ the empty set,~$(x + \dots x^k)$ by all sets with~${m(I) = 1}$, 
    and~$\binom{n}{1} \cdot (x + \dots x^{k-1})$ by all sets with $m(I) = 2$.
    Since~$\Chowaug_{U_{j-1,j}}(x) = A_j(x)$, $A_0(x) = 1$, and~$A_1(x)= 1$, this completes the proof.
\end{proof}

Write
\[
\Chow_{U_{k,n}} (x) = \sum_{i = 0}^{k-1} \underline{c }_{\ k,n}^{(i)} \ x^i
\quad \text{ and } \quad
\Chowaug_{U_{k,n}} (x) = \sum_{i = 0}^{k} c_{k,n}^{(i)} \ x^i \,.
\]
For a loopless matroid~$M$ on~$n$ elements and of rank~$k$, the coefficient of~$x^m$ in the Chow polynomial~$\Chow_M(x)$ and in the augmented Chow polynomial~$\Chowaug_M(x)$ are bounded by the corresponding coefficients in $\Chow_{U_{k,n}}(x)$ and in $\Chowaug_{U_{k,n}}(x)$, respectively, as proven in \Cite[Theorem 1.11]{Ferroni2024}. 
That is,
\[
[x^m ] \Chow_M (x) \leq \underline{c }_{\ k,n}^{(m)} 
\quad \text{and} \quad
[x^m ] \Chowaug_M (x) \leq c_{\ k,n}^{(m)} \,.
\]
A straightforward computation gives
\begin{align*}
    \underline{c }_{\ k,n}^{(1)}
    &= \sum_{i = 2}^k \binom{n}{\Delta \{1,i\}} 
    = 1 + \sum_{i = 2}^{k-1} \binom{n}{i} \,, \\
    c_{\ k,n}^{(1)}
    &= \sum_{i = 1}^k \binom{n}{\Delta \{i\}} 
    = \sum_{i = 0}^{k-1} \binom{n}{i} \,.
\end{align*}

The sequences $(c_{k,n}^{(1)})_{n\geq k}$ for fixed $k$ are sums of binomials, which are listed in the OEIS \Cite{oeis}.
For example, \href{https://oeis.org/A000125}{A000125} for $k=3$ and \href{https://oeis.org/A008859}{A008859} for $k=6$.

In the non-augmented case, we get a connection to \Dfn{Grassmannian permutations}. A permutation $w\in \Sym_n$ is called Grassmannian if it has at most one descent.

\begin{corollary}
    For $k\geq 2$, the coefficient $\underline{c }_{\ k,n}^{(1)}$ is the number of Grassmannian permutations of length~$n$ avoiding a (fixed) permutation $\sigma \in \Sym_k$ with $\des(\sigma) = 1$.
\end{corollary}
\begin{proof}
    This follows by \Cite[Theorem 3.3]{GilTomasko}, which provides the same formula as given here for $\underline{c }_{\ k,n}^{(1)}$.
\end{proof}

For fixed $k \leq 10$, the sequences $( \underline{c }_{\ k,n}^{(1)} )_{n\geq k}$ are listed in the OEIS,
e.g., \href{https://oeis.org/A000124}{A000124} for $k=3$,  \href{https://oeis.org/A050407}{A050407} for $k = 4$, \href{https://oeis.org/A027927}{A027927} for $k = 5$, and \href{https://oeis.org/A362193}{A362193} for $k=6$.

\medskip
The formulas that we get from \Cref{thm: multi Chows monomial} (with $x_i = x$) for the coefficients $\underline{c }_{\ k,n}^{(m)}$ and $c_{\ k,n}^{(m)}$ indicate, 
that the coefficients give the size of a set.
For instance, for~$m = 2$, we have
\begin{align*}
    \underline{c }_{\ k,n}^{(2)}
    &= 1 + \sum_{i = 3}^{k-1} \binom{n}{i} 
    + \sum_{i = 2}^{k-2} \binom{n}{i} 
    + \sum_{i = 2}^{k-3} \sum_{ j = 2 }^{k-i+1} \binom{n}{i,j,n-i-j} \,, \\
    c_{\ k,n}^{(2)}
    &= \sum_{i = 1}^{k-1} \binom{n}{i-1} 
    + \sum_{i = 1}^{k-2} \sum_{ j = 2 }^{k-i} \binom{n}{i-1,j,n+1-i-j} \,.
\end{align*}
Note that the sequences~$( \underline{c }_{\ k,n}^{(2)} )_{n\geq k}$ and $( c_{\ k,n}^{(2)} )_{n\geq k}$ are not listed in the OEIS, and the same applies for $m > 3$.
This leads to the following open problem:

\begin{open problem}
    Describe $\underline{c }_{\ k,n}^{(m)}$ and $c_{\ k,n}^{(m)}$ as sizes of sets of permutations of length $n$ according to certain restrictions.
\end{open problem}

\section*{Acknowledgment}
The author would like to thank Luis Ferroni for bringing Chow polynomials to her attention, sharing his conjecture, and providing valuable feedback during the writing of this paper.
She also wishes to thank Christian Stump for helpful discussions and his supervision throughout this research.
Additionally, the author acknowledges Hsin-Chieh Liao for pointing her to \Cite{liao2024equivariantgammapositivitychowrings} and offering insightful comments.

\printbibliography

\end{document}